\documentclass[a4paper,12pt]{amsart}
\usepackage{latexsym}
\usepackage{amssymb}
\usepackage{graphicx}
\usepackage{wrapfig}
\usepackage{amsthm}
\usepackage{float}
\usepackage{epsfig}
\usepackage{multirow}
\usepackage{caption}
\usepackage[toc,page]{appendix}
\usepackage{hyperref}
\RequirePackage{color}
\usepackage{amscd,enumerate}
\usepackage{amsfonts}

\input xy
\xyoption{all}

\usepackage{multicol}

\hypersetup{ colorlinks,
linkcolor=blue,
filecolor=green,
urlcolor=blue,
citecolor=blue }

\input xygraph

\usepackage{calrsfs}
\DeclareMathAlphabet{\pazocal}{OMS}{zplm}{m}{n}

\newcommand{\bK}{\mathbb{K}}
\newcommand{\bN}{\mathbb{N}}

\newcommand{\N}{\mathbb{N}}

\newcommand{\Dim}{{\rm dim}}
\newcommand{\Ker}{{\rm ker}}

\newcommand{\spa}{{\rm span}}
\newcommand{\rk}{{\rm rk}}
\newcommand{\erk}{{\rm erk}}
\newcommand{\ann}{{\rm ann}}
\newcommand{\Supp}{{\rm supp}}
\newcommand{\diag}{{\rm diag}}

\newcommand{\alg}{{\rm alg}}

\newtheorem{lemma}{Lemma}[section]
\newtheorem{corollary}[lemma]{Corollary}
\newtheorem{theorem}[lemma]{Theorem}
\newtheorem{proposition}[lemma]{Proposition}
\newtheorem{remark}[lemma]{Remark}
\newtheorem{definition}[lemma]{Definition}
\newtheorem{definitions}[lemma]{Definitions}
\newtheorem{example}[lemma]{Example}
\newtheorem{examples}[lemma]{Examples}

\newtheorem{notation}[lemma]{Notation}

\definecolor{turquoise2}{rgb}{0,0.898039,0.933333}
\definecolor{magenta}{rgb}{1,0,1}
\definecolor{olivedrab}{rgb}{0.419608,0.556863,0.137255}
\definecolor{purple2}{rgb}{0.568627,0.172549,0.933333}
\definecolor{amethyst}{rgb}{0.6, 0.4, 0.8}
\definecolor{ao(english)}{rgb}{0.0, 0.5, 0.0}
\definecolor{atomictangerine}{rgb}{1.0, 0.6, 0.4}
\definecolor{amber(sae/ece)}{rgb}{1.0, 0.49, 0.0}
\definecolor{alizarin}{rgb}{0.82, 0.1, 0.26}
\definecolor{auburn}{rgb}{0.43, 0.21, 0.1}
\definecolor{aqua}{rgb}{0.0, 1.0, 1.0}

\begin{document}
%

\subjclass[2010] {17D92, 17A60, 17A36, 15A99} \keywords{evolution algebra, natural family, decomposition, property (2LI), natural basis, evolution ideal, evolution subalgebra, extending property}

\title[Natural families in evolution algebras]{Natural families in evolution algebras}

\author[N. Boudi]{Nadia Boudi}
\address{Center CeReMAR, Laboratory LMSA, Faculty of Sciences, Mohammed V University,  Rabat, Morocco}
\email{nadia\_boudi@hotmail.com}

\author[Y. Cabrera ]{Yolanda Cabrera Casado}
\address{Departamento de Matem\'atica Aplicada, Universidad de M\'{a}laga, Spain}
\email{yolandacc@uma.es}

\author[M. Siles]{Mercedes Siles Molina}
\address{Departamento de \'{A}lgebra, Geometr\'{i}a y Topolog\'{\i}a, Universidad de M\'{a}laga, Spain}
\email{msilesm@uma.es}

\begin{abstract}
In this paper we introduce the notion of evolution rank and give a decomposition of an evolution algebra into its annihilator plus extending evolution subspaces having evolution rank one. This decomposition can be used to prove that in non-degenerate evolution algebras, any family of natural and orthogonal vectors can be extended to a natural basis. Central results are the characterization of those families of orthogonal linearly independent vectors which can be extended to a natural basis. 

We also consider ideals in perfect evolution algebras and prove that they coincide with the basic ideals.

Nilpotent elements of order three can be localized (in a perfect evolution algebra over a field in which every element is a square) by merely looking at the structure matrix: any vanishing principal minor provides one. Conversely, if a perfect evolution algebra over an arbitrary field has a nilpotent element of order three, then its structure matrix has a vanishing principal minor.

We finish by considering the adjoint evolution algebra and relating its properties to the corresponding in the initial evolution algebra.

\end{abstract}
\maketitle

\section{Introduction and preliminaries} \label{sec01}

The study of evolution algebras has its starting point in 2006 in the paper \cite{VT}, followed by the monograph \cite{T}. A systematic study of (arbitrary dimensional) evolution algebras was initiated in \cite{CSV1}, where a characterization of simple evolution algebras was accomplished. The problem of the classification of finite dimensional evolution algebras is a hard task. As far as we know, there have been classified (up to isomorphisms) the 2 and 3 dimensional evolution algebras  and the perfect non-simple four dimensional (with some mild restrictions); see \cite{CCY, CSV2, CGMMS1, CKS}. A useful tool in order to classify the four dimensional evolution algebras has been the notion of basic ideal: an ideal having a natural basis which can be extended to a natural basis of the whole algebra, introduced in \cite{CKS}. An interesting question related to basic ideals is when a vector in an evolution algebra is a natural vector (i.e., it is part of a natural basis of the evolution algebra). A natural follow up is when a family of orthogonal and independent vectors can be completed to a natural basis of the algebra. These have been the questions which have motivated initially this paper.

We have divided this article into five sections. The first one contains an introduction and preliminaries.  Concretely, we see that the linear span of a subset does not coincide necessarily with a subalgebra or ideal. Homomorphisms are recalled and  an example showing that automorphisms of evolution algebras  do not coincide in general with automorphisms of vector spaces. In fact, linear isomorphisms do not preserve, necessarily, natural bases. Evolution homomorphisms, defined in \cite{T} cannot coincide with algebra homomorphisms: An example of an algebra homomorphism such that the image has not the extension property is given. 

In Section 2 we consider natural families in evolution algebras. In the study of the classification of evolution algebras, an important aspect is when all the natural basis are essentially the same, i.e., all of them can be obtained from one by simply changing the order of the elements or multiplying each element by an scalar. We say in this paper that an evolution algebra having this property ``has a unique natural basis". It was proved in \cite[Theorem 4.4]{EL1} that when an evolution algebra $A$ is perfect (meaning $A^2=A$), then it has a unique natural basis. Corollary \ref{2LIB} finds a condition for an evolution algebra to have a unique natural basis: Having Property (2LI). This property is introduced in this paper for arbitrary evolution algebras and means that the square of two different elements in any natural basis are linearly independent. When an evolution algebra satisfies Condition (2LI), the dimension of $A^2$ must be at least 2. Obviously, any perfect evolution algebra having dimension bigger or equal than 2 satisfies this condition. For any natural number $n$ there exists an $n$-dimensional evolution algebra $A$ such that $A^2$ has dimension 2.

 One of the main results of the paper is Theorem \ref{nat-vec}, which characterizes when a vector is part of a natural basis. We call such elements ``natural vectors". These vectors can be discovered just by looking at the ranges of the squares of the elements ``in their support". Another important result in the paper is the decomposition of an evolution algebra into extending evolution subspaces (Theorem \ref{decomp}); these are, on the one hand, the annihilator of the evolution algebra, and vector subspaces generated by vectors in natural basis such that their squares are all linearly dependent. This decomposition depends on the natural basis, as Example \ref{EjDescomp} shows, but not completely, in the sense that the annihilator and the number of direct summands are always fixed. As a consequence, a block form for the change of basis matrices is obtained  (Corollary \ref{Ast}). We finish the section by providing a condition for a family of linearly independent and orthogonal vectors in a certain evolution subspace to have the extension property  (Proposition \ref{evsubs}). This condition will be used to prove that any family of orthogonal natural vectors can be always extended to a natural basis (Corollary \ref{CORevsubs}). 
 
Theorem 3.6 is the most important result in Section 3. It gives a tool to discover when there exist three elements in an evolution algebra whose product is zero. It happens precisely when the structure matrix has certain principal minors which are zero. In particular, there are nilpotent elements of order three if and only if there is a vanishing principal minor (Corollary \ref{cor:vanishingminor}).

Section 4 is devoted to the study of basic ideals and the relationship among the notions of simple, basic simple and evolution simple algebra, and prove that every ideal in a perfect evolution algebra is basic (Proposition \ref{BasicProp}).

Finally, in Section 5 we compare an evolution algebra and the evolution algebra obtained by transposing the structure matrix of the algebra (Proposition \ref{adjoint}). This algebra will be called the adjoint evolution algebra. We recover the decomposition of an evolution algebra given by Tian when considering persistent and transient elements in a natural basis. The connexion between this decomposition and the adjoint evolution algebra is given in Proposition \ref{AdjTrans}.


In this paper, $\bK$ will denote an arbitrary field and
the notation $\bK^\times$ 
will stand for $\bK\setminus\{0\}$. 
Also $\N^\times= \N\setminus\{0\}$.

An \emph{evolution algebra} over a field $\bK$ is a $\bK$-algebra $A$ which has a basis $B=\{e_i\}_{i\in \Lambda}$ such that $e_ie_j=0$ for every $i, j \in \Lambda$ with $i\neq j$. Such a basis is called a \emph{natural basis}.
From now on, all the evolution algebras we will consider will be finite dimensional and  $\Lambda$ will denote the set $\{1, \dots, n\}$.

Let $A$ be an evolution algebra with a natural basis $B=\{e_i\}_{i\in \Lambda}$.
Denote by $M_B=(\omega_{ij})$ the structure matrix of $A$ relative to $B$, i. e., $e_i^2 = \sum_{j\in \Lambda} \omega_{ji}e_j$.

\par \medskip \noindent
Take $u, v$ elements in $A$ and write $u= \sum_{i\in \Lambda} \alpha_i e_i$, $v= \sum_{i\in \Lambda}\beta_i e_i$, where $\alpha_i, \beta_i\in \bK$. 
Note that
\begin{equation*}
  uv= M_B \left(
              \begin{array}{c}
                \alpha_1 \beta_1 \\
                \vdots \\
                \alpha_n \beta_n \\
              \end{array}
            \right)= M_B \; D_u \left(
                                                              \begin{array}{c}
                                                                \beta_1 \\
                                                                \vdots \\
                                                                \beta_n \\
                                                              \end{array}
                                                            \right),
\end{equation*}
 where $D_u=\diag (\alpha_1, \ldots, \alpha_n)$.   \par \noindent

\medskip

\begin{definition}
\rm
Let $A$ be an evolution algebra, $B=\{e_i\}_{i\in \Lambda}$ a natural basis and $u=\sum_{i\in \Lambda}\alpha_ie_i$ an element of $A$.
The \emph{support of} $u$ \emph{relative to} $B$, denoted $\Supp_B(u)$, is defined as the set $\Supp_B(u)=\{i\in \Lambda\ \vert \  \alpha_i \neq 0\}$.   If $X \subseteq A$, we put $\Supp_B (X)= \cup_{x \in X} \; \Supp_B (x)$.
\end{definition}

\begin{notation}
\rm
Let $A$ be an evolution $\bK$-algebra. For a subset $X\subseteq A$ we will use $X^2$ to denote the set $\{xy \ \vert x, y \in X\}$ while $\spa(X)$ will stand for the $\bK$-linear span of $X$, $\alg(X)$ for the subalgebra of $A$ generated by $X$ and  $\langle X\rangle$ for the ideal of $A$ generated by $X$. 
\end{notation}

\begin{definition}
\rm
The \emph{rank} of a subset $X$ of an evolution algebra $A$, denoted by $\rk(X)$, is defined as the dimension (as a vector space) of  ${\rm span}(X)$. 
\end{definition}

\begin{example}
\rm
The ideal generated by a subset $X$ does not coincide necessarily with the linear span of $X$. For an example let $A$ be a four dimensional evolution algebra having a natural basis $B=\{e_1, \dots, e_4\}$ and product given by the structure matrix
$$
\begin{pmatrix}
1 & 1 & 0 & 0 \\
0 & 0 & 0 & 0 \\
0 & 1 & 0 & 1 \\
0 & 0 & 1 & 0
\end{pmatrix}.
$$
Then $\spa(\{e_1, e_2, e_3\})\neq \langle \{e_1, e_2, e_3\}\rangle = \spa(\{e_1, e_2, e_3, {e_3}^2\})= A$.
\end{example}
\medskip

Let $A$ and $A'$ be algebras over a field $\bK$.  A $\bK$-\emph{homomorphism} (or simply  \emph{homomorphism}) will be a $\bK$-linear map $f:A\to A'$. A $\bK$-\emph{algebra homomorphism} is a homomorphism $f:A \to A'$ which satisfies $f(ab)=f(a)f(b)$ for every $a, b\in A$. We will denote by $\rm{Hom}_\bK(A, A')$ and by $\rm{Hom}(A, A')$ the sets of homomorphisms and algebra homomorphisms, respectively, from $A$ to $A'$. When $A'=A$, we will speak about \emph{endomorphisms} and \emph{algebra endomorphisms}, respectively, and the corresponding sets will be denoted by $\rm{End}_\bK(A)$ and $\rm{End}(A)$. Note that $\rm{Hom}_\bK(A, A')$ and $\rm{Hom}(A, A')$ are $\bK$-vector spaces while $\rm{End}_\bK(A)$ and $\rm{End}(A)$ are  associative $\bK$-algebras. The subalgebras of  $\rm{End}_\bK(A)$ and $\rm{End}(A)$ consisting of those bijective maps are denoted by $\rm{Aut}_\bK(A)$ and $\rm{Aut}(A)$. 

Note that $\rm{End}_\bK(A) \supseteq \rm{End}(A)$. The following example shows that a strict  containment $\rm{Aut}_\bK(A) \supseteq \rm{Aut}(A)$ is possible.

\begin{example}
\rm
Let $A$ be an evolution algebra having a natural basis $\{e_1, e_2\}$ and product given by $e_1^2=e_1+e_2$ and $e_2^2=e_2$. Then the linear map $\varphi: A \to A$ given by $\varphi(e_i)=e_j$, for $\{i, j\}= \{1, 2\}$ is an automorphism of vector spaces but it is not an algebra automorphism as $\varphi(e_1^2)=\varphi(e_1+e_2) = \varphi(e_1)+\varphi(e_2)= e_2+e_1 \neq e_2=\varphi(e_1)^2$.
\end{example}

Homomorphisms that will be very useful are those preserving natural basis (i.e., those sending a natural basis to a natural basis of the image), and those preserving the extension property (i.e., sending a natural basis to a natural basis of the image). Not every homomorphism is so nice. 

The following is an example of a linear endomorphism sending a natural basis into a natural basis but which does not act in the same way when applied to any  natural basis.

\begin{example}
\rm
Consider the evolution algebra $A$ with natural basis $B=\{e_1, e_2, e_3\}$ and product given by $e_1^2=e_2^2=e_1$ and $e_3^2=0$. Let $\varphi: A \to A$ be the linear map such that $\varphi(e_1)=e_1+e_3$, $\varphi(e_2)=2e_2+e_3$, $\varphi(e_3)=e_3$. Then $\{\varphi(e_1), \varphi(e_2), \varphi(e_3)\}$ is a natural basis.

Take $B'=\{e_1+e_2, e_1-e_2, e_3\}$, which is a natural basis. Then $\{\varphi(e_1+e_2), \varphi(e_1-e_2), \varphi(e_3)\}=\{e_1+2e_2+2e_3, e_1-2e_2, e_3\}$ is not a natural basis.
\end{example}

Tian speaks in his book \cite{T}, about \emph{evolution homomorphisms} between evolution algebras $A$ and $A'$ as $\bK$-algebra homomorphisms satisfying that $\rm{Im}(f)$ (which turns out to be an evolution algebra) has the extension property (an evolution subalgebra $A'$ of an evolution algebra $A$ is said to have the \emph{extension property} if there exists a natural basis of $A'$ which can be extended to a natural basis of $A$; this definition was introduced in \cite{CSV1}). Note that when dealing with $\bK$-algebra homomorphisms, the image of every natural basis is a natural basis of the image. However, not every algebra homomorphism between evolution algebras is an evolution homomorphism, as the following example shows.

\begin{example}\label{ex: dim1}
\rm
Consider the evolution $\bK$-algebra $A$ with natural basis $\{e_1, e_2\}$ and product given by: $e_1^2 = -e_2^2=e_1+e_2$. Let $\varphi: A \to A$ be the linear map defined by $\varphi(e_1)=-\varphi(e_2)= e_1 + e_2$. Then, it is easy to see that $\varphi$ is an algebra homomorphism. Moreover, $\rm{Im}(\varphi) = \spa(\{e_1+e_2\})$, which has not the extension property (see \cite[Table in pg. 75]{CCY}).
\end{example}


\section{Natural families in an evolution algebra}

Let $A$ be an evolution algebra with natural basis $B$ and structure matrix $M$. Any  automorphism of $A$ maps a natural basis onto a generator set.  The automorphism group  of an evolution algebra has been considered in \cite{CGOT, EL1, T}, and also in \cite{CGMMS1}, where the authors determine the automorphism group of any  two-dimensional evolution algebra over an arbitrary field. 

Clearly, if we multiply each element of a natural basis $B$ by a nonzero scalar and reorder the elements, we obtain another natural basis $B'$ of $A$. The corresponding subgroup of the automorphisms group, $S_n\rtimes (\bK^\times)^n$, was described in \cite{CSV2} in the case of $3$-dimensional evolution algebras and in \cite{CKS} for an arbitrary dimension $n$. Thus, the following definition arise:

\begin{definition}
\rm
Let $A$ be an evolution algebra. We say that $A$ \emph{has a unique natural basis} 
if the subgroup of Aut$_\bK(A)$ consisting of those elements that map natural bases into natural bases is $S_n\rtimes (\bK^\times)^n$.
\end{definition}

By \cite[Theorem 4.4]{EL1} every perfect evolution algebra (an algebra $A$ is said to be \emph{perfect} if $A^2 = A$) has a unique natural basis.

In this section we  characterize when an evolution algebra has a unique natural basis. For that purpose we have to introduce the definition of extending natural family and natural vector.

\begin{definitions}
\rm
Let $A$ be an evolution algebra. Two elements $u$ and $v$ of an evolution algebra will be called  \emph{orthogonal} if $uv=0$. A family of vectors  $C$ is said to be an \emph{orthogonal family} if $uv=0$ for any $u, v\in C$, with $u\neq v$.

A family of pairwise orthogonal and linearly independent vectors $\{u_1, \ldots, u_r\}$ of $A$ which can be extended to a natural basis of $A$ will be
called an \emph{extending natural family}. 
If $\{u\}$ is an extending natural family for some $u \in A$, we will say that $u$ is a \emph{natural vector}. 
\end{definitions}

\begin{remark}\label{repeated}
\rm
Orthogonality does not imply linear independency, even if the evolution algebra is non-degenerate. For a simple example, consider  a nonzero element $e$ of an evolution algebra $A$ that satisfies $e^2=0$. Then for every scalar $\lambda$, $e$ and $\lambda e$ are orthogonal.   

For a non-degenerate example, consider an evolution algebra $A$ with a natural basis $\{e_1, e_2\}$ and product given by $e_1^2=e_1+e_2$ and $e_2^2 = -e_1-e_2$. Then $e_1+e_2$ and $-e_1-e_2$ are orthogonal and linearly dependent.  
\end{remark}

The  theorem that follows characterizes natural vectors in terms of {{their}} supports.

\begin{theorem}\label{nat-vec}
Let $A$ be an evolution $\bK$-algebra with natural basis $B= \{e_i\}_{i\in \Lambda}$ and let $u \in A$. Set $\Supp (u)= \{i_1, \ldots, i_r\}$. Then:
\begin{enumerate}[\rm (i)]
\item If $u^2 \neq 0$, then $u$ is a natural vector if and only if $\rk(\{e_{i_1}^2, \dots, e_{i_r}^2\})= 1$.
\item If $u^2 =0 $, then $u$ is a natural vector if and only if $e_{i_1}^2=\dots= e_{i_r}^2 = 0$. 
\end{enumerate}
\end{theorem}

\begin{proof}
Write $u= \sum_i \alpha_i e_i$, where $\alpha_i \in \bK$. 
\par \noindent
(i) Since $u^2 \neq 0 $, $\Dim(\Ker(M_B D_u)) \leq n-1$.
Thus $u$ is a natural vector if and only if $ \Dim \; \Ker \;(M_B D_u)=n-1$. Since $\Supp (u) =\{i_1, \ldots, i_r\}$, we get
\begin{equation*}
  \rk \;(M_B D_u)= \rk(\{e_{i_1}^2, \ldots, e_{i_r}^2\}) = 1,
\end{equation*}
 as desired. 
\par \noindent
(ii) Since $u \in \Ker(M_B D_u)$, $u$ is a natural vector  if and only if $M_B D_u=0$, that is,  $e_{i_1}^2= \ldots =e_{i_r}^2=0$.
\end{proof}

For an evolution algebra $A$, we recall that the \emph{annihilator} of $A$, denoted by $\ann(A)$, is the ideal of $A$ defined as 
$$\ann(A):=\{x \in A \ \vert \ xa=0 \ \text{for any}\ a\in A\}.$$

\begin{remark}
\rm
In an evolution algebra $A$ a vector $u$ such that $u^2=0$ is not necessarily an element in the annihilator of $A$. For an example, consider the evolution algebra having natural basis $\{e_1, e_2\}$ and product given by $e_1^2=e_1$, $e_2^2=-e_1$. Then $u=e_1+e_2$ has zero square but it is not in the annihilator of the algebra, which is zero.
\end{remark}

The  equivalences of the three conditions in Definition \ref{non-deg} can be obtained from  \cite[Definition 2.16]{CSV1}, \cite[Proposition 2.18]{CSV1}  and \cite[Corollary 2.19]{CSV1}.

\begin{definition}\label{non-deg}
\rm
An evolution algebra $A$ is said to be \emph{non-degenerate} if it satisfies the following equivalent conditions:

\begin{enumerate}[(i)]
\item There exists a natural basis $B$ of $A$ such that $b^2\neq 0$ for any $b\in B$.
\item Any natural basis $B$ of $A$ satisfies $b^2\neq 0$ for any $b\in B$.
\item $\ann(A)=\{0\}$.
\end{enumerate}
\end{definition}

As a corollary of Theorem \ref{nat-vec}, we have a characterization of having a unique natural basis.

\begin{corollary} \label{2LIB} Let $A$ be a non-degenerate evolution algebra over $\bK$. Then the following assertions are equivalent:
\begin{enumerate}[\rm (i)]
\item $A$ has a unique natural basis. 
\item There exists a natural basis $B$ such that for any $2$ different vectors $u$ and $v$ of $B$, $u^2$ and $v^2$ are linearly independent.
\end{enumerate}
\end{corollary}
\begin{proof}  Suppose that (i) holds true and let $u$, $v$ be two different vectors of  the natural basis $B$ of $A$.  Assume that $\rk \{u^2, v^2\}=1$. Being
%
$A$ non-degenerate implies $u^2 \neq 0$;  write $ v^2= \beta u^2$ for some nonzero $\beta \in \bK$; since $\bK$ has strictly more than three elements, we can find $\alpha\in \bK$ such that $\alpha$ is not a root of the polynomial $x^2+{1/\beta}$, i.e., $1+\alpha^2\beta\neq 0$. 
Then the element $w= u+ \alpha v$ satisfies $w^2= (1+\alpha^2\beta)u^2\neq 0$ and,  by Theorem \ref{nat-vec}, $w$ is a natural vector.  This contradicts the uniqueness of $B$.  The converse
(i) $\Rightarrow$ (ii) follows immediately from Theorem \ref{nat-vec}.
\end{proof}

The definition that follows is a generalization of the notion of Property (2LI) given in \cite[Definition 3.4]{CSV2} for three dimensional evolution algebras.

\begin{definition}\label{Def:PropmLI}
\rm
Let $A$ be an evolution algebra of dimension $n$. We say that $A$ has  \emph{Property} \rm{(2LI)} if  for any different vectors $e_i, e_j$ of a natural basis, the set  $\{e_i^2, e_j^2\}$ is linearly independent. 
\end{definition}

This definition is consistent because it does not depend on the selected natural basis, as follows from  Corollary \ref{2LIB}.

\begin{example}\label{Ex:DimDos}
\rm
For every natural number $n$ there exists an $n$-dimensional evolution algebra having Property (2LI) such that $\dim A^2=2$.
Let $\bK$ be an infinite field and let $A$ be the evolution $\bK$-algebra having basis $\{e_1, \dots, e_n\}$, for $n>2$, and product given by $e_1^2=e_1$, $e_2^2=e_2$, $e_i^2=e_1+ie_2$, where $i\in \{3, \dots, n\}$. Then $A$ has the property  (2LI).
\end{example}

\medskip
\begin{definitions}
\rm
Let $A$ be an evolution algebra. Any (linear) subspace $E$ of $A$  generated by {an} extending natural family will be called an \emph{extending evolution subspace}  of $A$. Such a family will be called an \emph{extending natural basis} of $E$.
 The \emph{evolution rank} of $E$ is defined by
\begin{equation*}
  \erk(E)= \Dim (\spa\{u_1^2, \ldots, u_r^2\}),
\end{equation*}
where $\{u_1, \ldots, u_r\}$ is an extending natural basis of $E$. Note that $\erk(E)=\dim (E^2)$.

Clearly, $\erk(E)$ does not depend on the choice of the extending natural basis of $E$.
\end{definitions}

\begin{theorem} \label{decomp} Let $A$ be an evolution algebra {{ and let $r= \dim A^2$.  Then}}
\begin{equation*}
  A= \ann(A) \oplus E_1 \oplus \ldots \oplus E_r,
\end{equation*}
where  $E_1, \ldots, E_r$ are extending evolution subspaces of $A$ satisfying $\erk (E_i)=1$ for all $i$ and if $i \neq j$, $E_i  E_j = 0$,   $\Dim \; ( E_i^2 +E_j^2)=2$. 
Moreover, if $A$ is non-degenerate, the decomposition is unique.
\end{theorem}
\begin{proof} Let $B$ be a natural basis of $A$. Taking into account that the annihilator of $A$ is the subalgebra of $A$ generated, as a vector space, by those elements in any natural basis whose square is zero (see \cite[Proposition 2.18 (i)]{CSV1}), we can write $B$ as
$B= B_0 \cup B_1 \cup \ldots \cup B_r$, where $B_0$ is a basis of $\ann (A)$,
\begin{equation*}
  \Dim(\spa \; \{e^2\ \vert \ e \in B_i\})=1, \;\; \mbox { for all } i \geq 1
\end{equation*}
and
\begin{equation*}
  \Dim(\spa \;\{u^2, v^2\})=2 \;\;\;  \mbox { if } u \in B_i, v \in B_j
\end{equation*}
for  all $i \neq j$ with  $1 \leq i, j \leq r$. Now let $E_i$ be the vector subspace generated by $B_i$, which is an extending evolution subspace, and then we get the desired decomposition. \par \noindent
Next suppose that $A$ is non-degenerate and let us show that  the above decomposition is unique. Write
\begin{equation*}
  A=  E'_1 \oplus \ldots \oplus E'_s,
\end{equation*}
where $E'_1, \ldots, E'_s$ are extending evolution subspaces of $A$ satisfying $\erk (E'_i)=1$;  $E'_i \; E'_j = 0$    and    $\Dim \; ( (E'_i)^2 +(E'_{j})^2)=2$, for all $i, j$ being $i \neq j$.
 Fix $t \in \{1, \ldots, s\}$ and let $B'$ be an extending natural basis of $E'_t$.  Let $u \in B'$  with
$\Supp (u)= \{i_1, \ldots, i_k \}$. Since $u^2 \neq 0$,   by Theorem \ref{nat-vec}, $\Dim(\spa\;\{e_{i_1}^2, \ldots, e_{i_k}^2\})= 1$. By the above construction, there exists $j_t$ such that
 $ \{e_{i_1}, \ldots, e_{i_k}\} \subseteq  E_{j_t}$. Since $\erk \; (E'_t)=1$, and $A$ is non-degenerate, we must have $B' \subseteq E_{j_t}$. This implies that $E'_t \subseteq E_{j_t}$.  Since
$\Dim \; ( (E'_t)^2 +(E'_{k})^2)=2$
for $t \neq k$, we infer that  $j_t \neq j_k$.  A classical argument shows that $r=s$ and $E'_t= E_{j_t}$ for all $t$.
\end{proof}

The following example shows that the above decomposition depends on the natural basis.
\begin{example}\label{EjDescomp}
\rm
Let $A$ be the evolution algebra with natural basis $\{ e_1, e_2, e_3\}$ and product defined by
\[ e_1^2= 0,\;\;  e_2^2 =e_2 \;\;\; \mbox {and } \;\;\; e_3^2=e_3.\]
Then $\{ e_1, e_1+e_2, e_3\}$ is a natural basis and the two bases generate  different decompositions for $A$ in the sense of Theorem  \ref{decomp} 
\[ A= \bK e_1 \oplus \bK e_2 \oplus \bK e_3= \bK e_1 \oplus \bK (e_1+e_2) \oplus \bK e_3.\]
\end{example}

\begin{corollary}\label{Ast}
Let $A$ be an evolution algebra and let $B=B_0 \cup B_1\cup \dots \cup B_r$ and $B'=B'_0 \cup B'_1\cup \dots \cup B'_r$ be two natural bases of $A$ given by two decompositions as in Theorem \ref{decomp}, where $B_0$ and $B'_0$ are bases of $\ann (A)$. Then, we can reorder the elements of $B$ and $B'$ so that
the change of basis matrix  has the following block form
\begin{equation*}
\left(
  \begin{array}{ccccc}
    * & * & * & \ldots & * \\
    0 & * & 0 &  \ldots&  0\\
    0 & 0 & * &  \ldots& 0 \\
    \vdots& \vdots & \vdots &  \ddots& 0 \\
    0 & 0 & 0 & \ldots & * \\
  \end{array}
\right).
\end{equation*}
\end{corollary}

\begin{remark}
\rm
The block form above means that the change of basis matrix sends the elements of $B_0$ to the linear span of the elements of $B'_0$ and, in general, the elements of $B_i$ are sent to the linear span of the elements of $B'_0\cup B'_i$.
\end{remark}

\begin{remark}
\rm
The block form in Corollary \ref{Ast} appears in \cite[Corollary 3.6]{EL2} for nilpotent evolution algebras.
\end{remark}

\begin{proposition} \label{evsubs} Let $A$ be an evolution algebra and let $E$ be an extending evolution subspace of $A$ with evolution rank one and such that $E \cap \ann (A)=\{0\}$.
 Let $C$ be a linearly independent orthogonal family  of $E$. Then  $C$ can be extended to a natural basis of $E$, which can be extended to a natural basis of $A$, if and only if $u^2 \neq 0$ for all $u\in C$.
\end{proposition}
\begin{proof} Let $B=\{e_1, \ldots, e_r\}$ be an extending natural basis of  $E$ and let $B'\subseteq A$ be such that $B\sqcup B'$ is a natural basis of $A$. For every $i\in \{1, \dots, r\}$ there exists $\lambda_i \in \bK ^\times$ such that $e_{i}^2= \lambda_i e_1^2$. The reason for $\lambda_i \neq 0$ is that $E \cap \ann (A)=\{0\}$ and \cite[Proposition 2.18 (i)]{CSV1}.
Consider the bilinear form $b$ defined  on $E$ by
\begin{equation*}
  b\; (\sum_{i=1}^r \alpha_i e_i, \sum_{i=1}^r \beta_i e_i)= \sum_{i=1}^r \lambda_i \alpha_i \beta_i, \;\; \alpha_i, \beta_i \in \bK.
\end{equation*}
Note that $b$ is symmetric, non-degenerate and $u v= b(u,v) e_1^2$ for all $u,v \in E$. Suppose that $u^2 \neq 0$ for all $u \in C$. Set $V= \spa \;(C)$ and $b'=b|_V$. We are going to prove that $b'$ is non-degenerate. Suppose that $b'(v,.)=0$ for some $v \in V$.
Then $v u=0$ for all $u \in C$. Write $v= \sum_{i=1}^s \alpha_i u_i$, where $u_i \in C$; then $\alpha_i u_i^2=0$ for all $i$. This implies that $v=0$. Whence $b'$ is non-degenerate and we have $V \oplus V^\perp =E$. Therefore
$C$ can be extended to an orthogonal  basis $C'$ of $E$. 
Note that  $C' \sqcup B'$ is a natural basis of $A$ and  $C'$ is a natural basis of $E$. The converse is obvious.
\end{proof}

A consequence of Proposition \ref{evsubs} follows.

\begin{corollary}\label{CORevsubs}
Let $A$ be a non-degenerate evolution algebra and let $B'=\{u_1, \ldots, u_r\}$ be a family of orthogonal vectors such that $u_i$ is a natural vector of $A$ {\color{blue}{ for every $i$}}.
Then $B'$ can be extended to a natural basis of $A$.
\end{corollary}
\begin{proof} By Theorem \ref{decomp} we get a (unique) decomposition $A= A_1 \oplus \ldots \oplus A_s$. Take $u\in B'$. Observe that $u^2 \neq 0$ since $u$ is a natural vector and $A$ is non-degenerate. By Theorem \ref{nat-vec} (i), there exists $i\in \{1, \dots, s\}$ such that $u\in A_i$. This provides a decomposition $B'= \cup_{i=1}^t B_i$,  with $B_i \subseteq A_i$. Since $B'$ is linearly independent (because $u^2\neq 0$ for each $u\in B'$ and the elements of $B'$ are orthogonal), each $B_i$ is a linearly independent orthogonal family of $A_i$. By Proposition \ref{evsubs}, $B_i$ can be extended to a natural basis of $A_i$.Therefore, $B'=\cup_{i=1}^t B_i$ can be extended to a natural basis of $A$, as required.

\end{proof}



\section{Orthogonal elements and nil elements in an evolution algebra}


 Let $A$ be an evolution algebra over  $\mathbb K$. For every natural number $k $ and $a \in A$, we write $a^1=a$ and  $a^k= aa^{k-1}$. Moreover, we denote $A^1= A^{\langle 1\rangle }=A$, $A^{k+1}= \sum_{i=1}^k A^iA^{k+1-i}$ and $A^{\langle k+1\rangle }= A^{\langle k\rangle } A$.  The algebra $A$ is said to be \emph{nilpotent} if $A^k=0$ for some $k$ and  it is said to be \emph{right nilpotent}{ if $ A^{\langle k\rangle }=0$ for some $k$.  An element $a$ of $A$ is said to be \emph{nil} if there exists $k \in \bN^\times$ such that $a^k=0$ and the algebra $A$ will be called \emph{nil} if every element is nil.
 
 \begin{remark}\label{nilpot}
 \rm
 Let $A$ be an evolution algebra. 
 \begin{enumerate}[\rm (i)]
 \item Then $A$ is nilpotent if and only if $A$ is right nilpotent if and only if $A$ is nil, if and only if there exists a basis $B$ in $A$ such that $M_B$ is strictly upper triangular.
To see this, use that a commutative algebra is nilpotent if and only if it is right nilpotent (this was shown in \cite{ZSSS}). On the other hand, the equivalence among the other three conditions is proved in \cite[Theorem 2.7]{CGOT}.
\item Assume $A$ nilpotent having dimension $n$. Then, being $M_B$  strictly upper triangular (for some natural basis $B$) implies $A^{<n+1>}=0.$ 
\item The evolution algebra $A$ is right nilpotent of order $k$, i.e., $A^{<k>}=0$, if and only if every natural basis $B$ can be reordered in such a way that $M_B$ is strictly upper triangular, and the row $i$ is equal to zero for every $i \geq k-1$. The case $k=3$ yields a nice characterization.
\end{enumerate}
\end{remark}

  \begin{lemma}\label{nilp-3} Let $A$ be an $n$-dimensional evolution algebra and $B$ a natural basis. Then $A^3=0$ if and only if for every $i\in \{1, \dots, n\}$, either the $i$-th row of the matrix $M_B$ is zero or the $i$-column is zero. More concretely, if the dimension of $\ann(A)=r$, then we may reorder the basis $B$ so that every row  from the $(r+1)$-th to the last one in $M_B$ is zero. \end{lemma}
\begin{proof} Note that $A^3=0$ implies $A^2 \subseteq \ann(A)$. Take this into account for the rest of the proof. The first statement follows from (iii) in Remark \ref{nilpot}. To prove the second assertion, we reorder $B$ in such a way that the first $r$-elements in $B$ are in the annihilator of $A$.
\end{proof}

  Let $A$ be an evolution algebra. Put $ann^1(A)=ann (A)$ and for $i \geq 2$, denote by $\ann^i(A)$ the set 
   \[\ann^{i}(A)= \spa \lbrace e \in B\ \vert \ e^2 \in \ann^{i-1} (A)\rbrace. \]
  If $A$ is nilpotent, then there exists an integer  $r$ such that  $A= \ann^r(A)$.  Let $r$ be the lowest natural number satisfying this equality. 
   Put $n_1= \dim(\ann(A))$ and
   $$n_i=\dim \left(\ann^i(A)/(\ann^{i-1}(A))\right) =\dim (\ann^i(A))- \dim (\ann^{i-1}(A)).$$  
  The \emph{type}  of the nilpotent algebra $A$, given in \cite{EL2},  is the ordered sequence $ [n_1, \ldots, n_r]$. 
   
Observe that if  $A$ is nilpotent of type $[n_1, \ldots, n_r]$, then the index of right nilpotency is exactly $r+1$. This follows taking into account the chain of annihilators that follows, which is strict and stabilizes:
$$\ann^1(A)  \subseteq \dots \subseteq \ann^r(A)=A.$$
Thus, a nilpotent evolution algebra for which the right index of nilpotency is the highest possible must have type $[1, \ldots, 1]$. \\

\begin{remark}
\rm
 It is well known that a degenerate evolution algebra  may not have a unique natural basis. Indeed, one may construct easily different natural bases by using elements of the annihilator  (see for instance Corollary \ref{Ast}).  
 
 Even if $A^{\langle n+1\rangle}=0$ we cannot assure that $A/\ann(A)$ has a unique basis. For an example, take $A$ with natural basis $\{e_1, e_2, e_3\}$ such that $e_1^2 =0$, $e_2^2=e_1$, $e_3^2=e_1+e_2$. Then $A^{\langle 4\rangle}=0$; moreover, $\{\overline{e_2}, \overline{e_3}\}$ and $\{\overline{e_2}, \overline{e_2+e_3}\}$ are natural basis of $A/\ann(A)$, so this algebra has not a unique natural basis. Note that  this algebra has type $[1, \ldots, 1]$.

There exists similar examples  for different types: Let $ A$ be a  nilpotent evolution algebra  having type $[1,n_2, \ldots, n_r]$ with $n_2 \geq 2$, then $A/ann(A)$ does not have a unique basis. Indeed, Let $B$ be a natural basis of $A$ and let $e, f$ be two elements of $B \cap \ann^2(A)$ such that $\overline{e}$ and $\overline{f}$ are linearly independent elements of $A/\ann(A)$. Then, $e^2 , f^2 $ lie in  $\ann(A)$ and  are linearly dependent. By Theorem \ref{nat-vec}, the vector
 $\overline{e+f}$ is  natural in $A/\ann(A)$.
  \end{remark}
  
  It is well known and easy to see that the index of nilpotency is in general greater than the order of right nilpotency in an algebra. Indeed, the index of nilpotency depends also  on the index of solvability. Recall that for an  algebra $A$, the solvable subalgebras are defined by $A^{[1]}= A$ and $A^{[k+1]}=A^{[k]} A^{[k]}$.    Next we provide an example  of a nilpotent evolution algebra $A$ such  that $A^3 \neq A^{[2]}$. 

  \begin{example}
  \rm
  Let $A$ be an evolution algebra of dimension $4$ and natural basis $\lbrace e_1, \ldots, e_4\rbrace$. Define the product on $A$ by the relations
  \[e_1^2=-e_2^2=e_4, \;\;\; e_3^2=e_1+e_2, \;\; e_4^2=0\]
  Then it is straightforward to check that $A^{[2]}=A^2= \spa \lbrace e_1+e_2, {e_4}\rbrace$,  $A^3= \spa \lbrace e_4\rbrace$ and that $A^{\langle 4 \rangle} =  \lbrace 0 \rbrace$.
  \end{example}
    On the other hand, it was shown in \cite{CSV2} that the number of nonzero entries of the structure matrix can characterize the evolution algebra when it is perfect, equivalently, when the structure matrix is nonsingular. Next we will investigate a possible generalization,  the property of having vanishing minors for the structure matrix.\par \noindent
Denote by $M_B$ the structure matrix of $A$ relative to a natural basis $B$.   It is easy to see that $M_B$ is singular if and only if there exist non-trivial orthogonal elements in $A$. This is also equivalent to the existence of  $u, v \in A$  such that $\Supp (u)= \Supp (v)$ and $uv=0$.  One also can check that if  $\bK$ is algebraically closed,  the structure matrix $M_B$ of {$A$} is singular if and only if there exists $u \in A$ such that $u^2=0$ (see \cite{CGOT}). We will be more precise in Proposition \ref{minors}. First, we introduce the following notation. 

\begin{notation}\label{not:Delta}
\rm
Let $A$ be an evolution algebra with natural basis $B=\{e_i\}_{i\in \Lambda}$ and structure matrix $M_B=(\omega_{ij})$. Take $\Gamma, \Omega \subseteq \Lambda$ \begin{enumerate}
  \item[\rm (a)] Assume $\vert \Gamma \vert \geq \vert \Omega\vert$; for any $\Delta \subseteq \Gamma$ with $\vert \Delta \vert = \vert \Omega\vert$, we denote by $M_\Delta=(w_{ij})$, where $i\in \Delta$ and $j\in \Omega$.
  \item[\rm (b)] Assume $\vert \Gamma \vert \leq \vert \Omega\vert$; for any $\Delta \subseteq \Omega$ with $\vert \Delta \vert = \vert \Gamma\vert$, we denote by $M_\Delta=(w_{ij})$, where $i\in \Gamma$ and $j\in \Delta$.
\end{enumerate}
\end{notation}

\begin{theorem}\label{minors} Let $A$ be a finite dimensional perfect evolution algebra over a field $\bK$ having a natural basis $B=\{e_i\}_{i\in \Lambda}$, and denote by $M_B=(\omega_{ij})$ the structure matrix of $A$ relative to $B$.
Then the following assertions are equivalent:
\begin{itemize}
  \item[(i)] There exist $u, v, w \in A$ such that $u (v w)=0$, where $ \Gamma : =\Supp (u)$ and $\Omega:= \Supp (v)= \Supp(w)$.
  \item[(ii)] There exist $\Gamma, \Omega \subseteq \Lambda$ such that $\vert M_\Delta\vert=0$ for every $\Delta$ as in Notation \ref{not:Delta}.
\end{itemize}
\end{theorem}
\begin{proof} Suppose that (i) holds.  To prove (ii) set $$u= \sum_{i\in \Gamma} \alpha_i e_i,\;\;\; v= \sum_{i\in \Omega} \beta_i e_i, \;\;\; w= \sum_{i\in \Omega} \gamma_i e_i.$$ Then
\begin{equation*}
u (vw)=\sum_{k \in \Gamma}  \sum_{i \in \Omega} \alpha_k \beta_i \gamma_i w_{ki} e_k^2=0.
\end{equation*}                                                                                                                      
But $M_B$ is nonsingular and $\alpha_k\neq 0$ for every $k\in \Gamma$. Hence, for all $k \in \Gamma$ we have $\sum_{i \in \Omega} \beta_i \gamma_i w_{ki}=0$. Since $\beta_i\gamma_i\neq 0$ for every $i\in \Omega$, we have  $\vert M_\Delta\vert=0$ for every $\Delta$ as in (a) or (b) in Notation \ref{not:Delta}.  
\par 

\noindent
Next, to prove that $(ii)$ implies $(i)$ suppose that there exist $\Gamma, \Omega\subseteq \Lambda$, with $\vert \Gamma \vert \geq \vert \Omega\vert$ such that $\vert M_\Delta \vert=0$ for every $\Delta$ as in the statement.
Then there exist scalars $\alpha_j \in \bK\setminus\{0\}$  such that $\sum_{j\in \Omega} \alpha_j w_{kj}=0$ for all  $ k \in \Gamma$ (if $\alpha_j=0$ for some $j$, then we change the set $\Omega$ by eliminating $e_j$).   Therefore, we infer that
\begin{equation*}
  \sum_{k \in \Gamma} \sum_{j\in \Omega} \alpha_j w_{kj} e_k^2=0.
\end{equation*}
This implies that
\begin{equation*}
   \sum_{t \in \Gamma} e_t \; (\sum_{k=1}^n  \sum_{j\in \Omega} \alpha_j w_{kj} e_k)=0,
\end{equation*}
that is,
\begin{equation*}
  ( \sum_{t \in \Gamma} e_t ) (  \sum_{j\in \Omega} \alpha_j e_j^2)=0.
\end{equation*}

We get the desired conclusion by considering $u= \sum_{t \in \Gamma} e_t$, $v= \sum_{j\in \Omega} \alpha_j e_j$ and $w= \sum_{j\in \Omega} e_j$.

If $\vert \Gamma \vert \leq \vert \Omega\vert$ we proceed in a similar way.
\end{proof}

If $A$ is not perfect, then the result is not true, as shown in the example that follows.

\begin{example}
\rm
Consider a two-dimensional evolution algebra $A$ with a basis $B=\{e_1, e_2\}$ and product given by $e_i^2=e_1+e_2$ for $i=1, 2$. Then the elements $u=e_1-e_2$, $v=w=e_1$ satisfy that $u(vw)=0$. Note that  in this case $\Gamma=\{1, 2\}$, $\Omega=\{1\}$ and  for every possible $\Delta$ we have $M_\Delta = (1)$, which has nonzero determinant.
\end{example}

\begin{corollary}\label{cor:vanishingminor} Let $A$ be a perfect evolution algebra over a  field $\bK$ having a natural basis $B=\{e_i\}_{i\in \Lambda}$ and structure matrix $M_B=(\omega_{ij})$. Consider the statements:
\begin{itemize}
  \item[(i)] There exists $u \in A$ such that $u^3=0$.
  \item[(ii)] $M_B$ has a vanishing principal minor.
\end{itemize}
Then (i) implies (ii). If any element in $\bK$ is an square, then (ii) implies (i). 
\end{corollary}

\begin{proof}
Take a nonzero $u\in A$ and denote by $\Gamma$ the support of $u$ relative to $B$. Write
 $u=\sum\limits_{j\in \Gamma} \alpha_je_j$; then $u^2=\sum\limits_{j\in \Gamma}\alpha_j^2e_j^2=\sum\limits_{k\in \Lambda}\left(\sum\limits_{j\in \Gamma}\alpha_j^2\omega_{kj}\right)e_k$ and 
 $u^3=	\left(\sum\limits_{l\in \Gamma}\alpha_le_l\right)\left(\sum\limits_{k\in \Lambda}\left(\sum\limits_{j\in \Gamma}\alpha_j^2\omega_{kj}\right)e_k\right) =\sum\limits_{k\in \Gamma} \alpha_k\left(\sum\limits_{j\in \Gamma}\alpha_j^2\omega_{kj} \right)e_k^2$.

(i) $\Rightarrow$ (ii). Assume $u^3=0$. Then, taking into account the previous computation and that $\{e_i^2\}_{i\in \Lambda}$ is a linearly independent set (because the algebra $A$ is perfect) we have  $\sum\limits_{j\in \Gamma}\alpha_j^2\omega_{kj} =0$ for each $k \in \Gamma$. If we interpreted $(\alpha^2_j)$ as the nontrivial solution of a linear system having matrix of coefficients $(\omega_{kj})_{k, j \in \Gamma}$, then the determinant of this  matrix has to be zero. This shows (ii).

(i) $\Rightarrow$ (ii). Now we suppose that every element in $\bK$ is an square. Let $M_\Delta=(\omega_{ij})$, where $i, j \in \Gamma\subseteq \Lambda$, be such that $\vert M_\Delta\vert=\vert(\omega_{ij}) \vert=0$. Being zero this determinant implies that the vectors $\{\sum_{k\in \Gamma}\omega_{kj}e_k \ \vert \ j\in \Gamma\}$ are linearly dependent and, therefore, there exist 
$\{\beta_j\in \bK \ \vert \ j\in \Gamma \}$, where some $\beta_j$ is nonzero, such that $\sum_{j\in \Gamma}\beta_j\left(\sum_{k\in \Gamma}\omega_{kj}e_k\right)=0$. Since $\{e_k \ \vert \ k\in \Gamma\}$ are linearly independent, then $\sum_{j\in \Gamma}\beta_j \omega_{kj}=0$ for every $k\in \Gamma$. 

Take $\alpha_j\in \bK$ such that $\alpha_j^2=\beta_j$. For $u=\sum_{j\in \Gamma}\alpha_je_j$, use the computations we did above to see that $u^3= \sum\limits_{k\in \Gamma} \alpha_k\left(\sum\limits_{j\in \Gamma}\alpha_j^2\omega_{kj} \right)e_k^2= \sum\limits_{k\in \Gamma} \alpha_k\left(\sum\limits_{j\in \Gamma}\beta_j\omega_{kj} \right)e_k^2$, which is zero.
\qedhere

\end{proof}

\section{Ideals in perfect evolution algebras}  

 Simple evolution algebras were thoroughly investigated in 
 \cite{VT}, \cite{CSV1} and in \cite{CKS}. In particular, it was shown in \cite{CSV1} that a finite dimensional evolution algebra $A$ is simple if and only if the structure matrix of $A$ relative to any basis $B$, is nonsingular and  $B$ cannot be reordered in such a way that the corresponding structure matrix has the following block form
\begin{equation*}
  \left(
    \begin{array}{cc}
      W & U \\
      0 & Y \\
    \end{array}
  \right).
\end{equation*}

Following \cite[Definitions 2.1]{CKS}, we say that an ideal $I$ of an evolution algebra $A$ with a natural basis $B$ is a \emph{basic ideal relative to} $B$ if $I$  has a natural basis consisting of vectors from $B$. If $A$ has no nonzero proper basic ideals relative to any natural basis $B$, then we say that $A$ is \emph{basic simple}.
If this happens, then  $A$ does not contain a proper evolution subalgebra  having the extension property. 
When the algebra $A$ is perfect, then an ideal $I$ is a basic ideal relative to a basis if and only if it is a basic ideal relative to any natural basis of $A$ (see \cite[Lemma 2.3]{CKS}). In this case we simply say that $I$ is a basic ideal

Next we compare the notions of simplicity and basic simplicity. The use of the support is fundamental to prove the results that follows.

\begin{example}
\rm
This is an example of an evolution algebra which is basic simple but not simple. Let $A$ be the two-dimensional evolution algebra with natural basis $\{e_1, e_2\}$ and product given by $e_1^2=-e_2^2=e_1 + e_2$. Then, the only {\color{blue}{proper}} ideal of $A$ is the one generated by $e_1+e_2$, which is not a natural vector by Theorem \ref{nat-vec} (ii).
\end{example}

\begin{proposition} \label{BasicProp}
Suppose that $A$ is a perfect evolution algebra. Then every nonzero ideal is a basic ideal. 
\end{proposition}
\begin{proof} Let B=$\{e_1, \dots, e_n\}$ be a natural basis of $A$.  Let $J$ be an ideal of $A$ and let $u \in J$. If $i \in \Supp (u)$, then $e_i u \in J$; this implies $e_i^2 \in J$. Now we prove that $e_i \in J$. Let $\Supp (J)=\{i_1, \dots, i_s\}$ and write $e_{i_j}^2= \sum_{k=1}^s a_{kj}e_{i_k}$. We claim that the determinant of the matrix $(a_{kj})$ is different from zero because $A^2=A$ implies that $\{e^2_1, \dots, e^2_n\}$ is a linearly independent set; in particular $\{e^2_{i_1}, \dots, e^2_{i_s}\}$ is linearly independent. Therefore $e_{i_j}\in \spa \left(\{e^2_{i_1}, \dots, e^2_{i_s}\}\right)\subseteq J$.
In other words, $J= \spa \left(\{ e_ j\ \vert \ j \in \Supp (J) \}\right)$, which means that it is a basic ideal. 
\end{proof}

\begin{remark}
\rm

The statement $A$ simple if and only if $A$ basic simple (for $A$ perfect) was proved in \cite[Proposition 2.7]{CKS} and also follows from \cite[Corollary 4.6]{CSV1}.
\end{remark}

Next we provide two examples which show that when the algebra $A$ is not perfect, the three kinds of simplicity (basic simple relative to a natural basis, simple and basic  simple) are different. 

\begin{example}
\rm
Let $A$ be a three dimensional evolution algebra over a field $\bK$ and assume that the structure matrix of $A$ relative to a natural basis $B$ is
\begin{equation*}
  M_B=\left(
     \begin{array}{ccc}
       0 & 1 & 1 \\
       1 & 0 & 1 \\
       1 & 1 & 2 \\
     \end{array}
   \right).
\end{equation*}
Then $A$ is not perfect, as $A^2$ has dimension 2. It can be  proved, with some computations, that $A^2$ has not a natural basis. On the other hand, it is easy to check that  $A$ is basic  simple relative to $B$. Since $A$ has a unique natural basis, $A$ is basic simple. However,  $A$ is not simple since $M_B$ is singular (has zero determinant). Moreover, if $\bK$ has characteristic $5$, then it can be proved that  $A^2$ has not a natural basis. The reason is that $A^2$ is an evolution algebra if and only if the polynomial $x^2+3x+1$ has $2$ distinct roots in $\bK$, and this happens under the assumption of having characteristic 5.
 \end{example}

\begin{example}
\rm
Let $A$ be the three dimensional evolution algebra with natural basis $B= \lbrace e_1, e_2, e_3\rbrace$ and assume that its structure matrix  with respect to  $B$ is
\begin{equation*}
  M_B=\left(
     \begin{array}{ccc}
       1 & 1 & 2 \\
       1 & 1 & 2 \\
       1 & 1 & 2 \\
     \end{array}
   \right).
\end{equation*}
It is easy to check that $A$ is basic simple with respect to $B$.  Now observe that  $B'= \lbrace e_1+e_2, e_1-e_2, e_3\rbrace$ is a natural basis of $A$, the corresponding structure matrix is 
\begin{equation*}
  M_{B'}=\left(
     \begin{array}{ccc}
       2& 2 & 2 \\ 
       0 & 0 & 0 \\
       2 & 2 & 2 \\
     \end{array}
   \right)
\end{equation*}
Moreover, the ideal $\langle \lbrace e_1+e_2, e_3\rbrace \rangle$ is basic relative  to $B'$.  In particular, $A$ is not basic simple.
\end{example}

\section{Algebraically persistent elements}

 The notions of algebraically persistent generators and  algebraically transient generators of an evolution algebra were defined  in \cite[p.42]{T}. Here we shall deal with persistent and transient generators with respect to a basis.    Let $B=\{e_i\}_{i\in \Lambda}$ be  a natural basis. We say that $e_i$ is \emph{algebraically persistent relative to $B$} if  the evolution subalgebra  it generates is basic simple and has the extension property relative to $B$, i.e., it has a natural basis consisting of vectors of $B$. Otherwise, $e_i$ is said to be \emph{algebraically transient relative to $B$}.  
These elements give rise to a decomposition which may help in analyzing the dynamical behavior of the evolution algebra, if we consider its hierarchical decomposition \cite{T, T2}.

We have to be careful with the concept of algebraic transiency and persistency  and the decomposition given in \cite{T}.   The examples that follow clarify in which sense.

\begin{examples}
\rm
A natural vector can generate a non simple evolution subalgebra. Take the evolution algebra $A$ with natural basis $\{e_1, e_2, e_3\}$ and product given by $e_1^2 =e_3^2= e_1+e_2$; $e_2^2= -e_1-e_2$. The subalgebra generated by  $e_1$ is $\bK e_1+\bK e_2$, which is not simple as $e_1+e_2$ generates the ideal $\bK (e_1+e_2)$.

A natural vector can generate a simple evolution subalgebra not having the extension property. Consider $A$ as the evolution algebra having a natural basis $\{e_1, e_2, e_3\}$ and product given by $e_1^2=e_2+e_3$; $e_2^2= e_2-e_3$; $e_3^2=e_1-e_2+e_3$. Then the subalgebra generated by $e_1$ is $\bK e_1 \oplus \bK (e_2+e_3)$, which is an evolution subalgebra, simple as an algebra, and not having the extension property.
\end{examples}

  \begin{example}
  \rm
  Consider the four dimensional evolution algebra $A$ with basis $ \lbrace e_1, \ldots, e_4\rbrace$ such that $$e_1^2= e_2+e_3+e_4, \;\;  e_2^2= e_1, \;\;  e_3^2=e_4^2= -1/2 \;e_1$$
  Then the linear subspaces $\spa\lbrace e_1, e_2, e_3+e_4 \rbrace$ and $ \spa \lbrace e_1, e_3, e_2+e_4\rbrace$ can be both seen as evolutions subalgebras having the extension property, containing $e_1$ and having minimal dimension. Observe that they are both basic simple.
  \end{example}
  
Next we provide an example showing that the number of transient elements in a natural basis depends on the natural basis itself.

\begin{example}\label{number-transient}
\rm
Let $A$ be the evolution algebra with natural basis $B= \{ e_1, e_2, e_3 \}$ and structure matrix with respect to $B$ given by
\begin{equation*}
  M_B=\left(
     \begin{array}{ccc}
       1 & 1 & 1 \\
       1 & 1 & 1 \\
       1 & 1 & 0 \\
     \end{array}
   \right).
\end{equation*}
Then, it is easy to see that $B'= \{ e_1+e_2, e_1-e_2, e_3\}$ is also a natural basis of $A$ and the corresponding structure matrix is
\begin{equation*}
  M_{B'}=\left(
     \begin{array}{ccc}
       2 & 2 & 1 \\
       0 & 0 & 0 \\
       2 & 2 & 0 \\
     \end{array}
   \right).
\end{equation*}
Observe that $e_1, e_2, e_3$ are algebraically transient relative to $B$, while $e_1-e_2$ is algebraically transient relative to $B'$, and $e_1+e_2$  and $e_3$ are algebraically persistent relative to $B'$.  Note that the number of algebraically persistent/transient elements depends on the natural basis. \\
On the other hand,  observe that, in the sense of \cite{T},  $e_1, e_2$ and $e_1-e_2$ are algebraically transient, while $e_1+e_2$ and $e_3$ are algebraically persistent.  In particular,  $e_1 +e_2$ is an element of the algebra generated by $e_3$, hence if $E$ is the linear space spanned by transient elements of $B$, it intersects the algebra generated by the persistent element $e_3$.
\end{example}

\begin{remark}
\rm
An evolution subalgebra/ideal of an evolution algebra $A$ can have the extension property relative to a natural basis of $A$ but not relative to any natural basis of $A$. For an example, consider the evolution algebra $A$ in Example \ref{number-transient} and the ideal $I$ generated by $e_3$. A natural basis of $I$ is $C=\{e_3, e_1+e_2\}$, which has not the extension property relative to $B=\{e_1, e_2, e_3\}$, while $C$ has the extension property relative to $B'=\{e_1+e_2, e_1-e_2, e_3\}$.
\end{remark}

Let $A$ be an evolution algebra. An \emph{evolution subalgebra} is a subalgebra with a natural basis. An evolution algebra is said to be irreducible  if it cannot be written in the form $A= A_1 \oplus A_2$, where $A_1$ and $A_2$ are two proper evolution subalgebras, equivalently evolution ideals, equivalently ideals (see \cite[Lemma 5.2]{CSV1}).

Now we recall the decomposition given by Tian in \cite[Theorem 11]{T}. 
Let $A$ be an irreducible evolution algebra. The notation $A' \dot{+} E$ will be used for the direct sum of a subalgebra $A'$ and a vector subspace $E$ of $A$.  
Assume that $B=\{e_i\}_{i\in \Lambda}$ is a natural basis. Then we classify elements in $B$ as algebraically persistent and algebraically transient (relative to $B$).  We have

\begin{equation}\label{descompForm}
 A= A_1 \oplus \ldots \oplus A_n  \; \dot{+} \; E 
 \end{equation}
where each $A_i$ is a simple evolution subalgebra of $A$ having the extension property relative to $B$ (i.e., $A_i$ is generated by a persistent element of $B$), and $E$ is the subspace spanned by the algebraically transient generators (relative to $B$). Since we deal with transiency and persistency with respect to the basis $B$, it is clear in our case that $E \cap( A_1 \oplus \ldots \oplus A_n) = \lbrace 0\rbrace$.   Note that the basis $B$ is the union of suitable bases of $E$ and the $A_i$'s. The linear space $E$ is called the \emph{$0$th transient space of} $A$. The decomposition depends on the basis, so it is not in general unique (see for instance Example \ref{number-transient}).  This provides a decomposition of $A$ which is called the  \emph{ $0$th decomposition of} $A$. By considering $E$ as an evolution algebra we may repeat the process and find the 0th decomposition of $E$. The resulting decomposition is called the \emph{ $1$th decomposition of} $A$. An induction process associates a hierarchy to  $A$  (see \cite[p. 46]{T} and \cite{T2}).
\\

Recall the following definitions from \cite{CSV1}. 
Let $A$ be an evolution algebra with natural basis $\{ e_i\}_{i\in \Lambda}.$  Put 
\[ D^1(i) = \Supp (e_i^2); \quad \hbox{and for $k\geq 2$}, \quad D^k(i)= \Supp \left(\{e_j^2 \ \vert \ e_j \in D^{k-1}(i)\}\right).\]
The elements of $D^k(i)$ are called the \emph{$k$th-generation descendants of} $i$. We will refer to  $D(i)=\cup_{k\in \N^\times}D^k(i)$ as the set of \emph{descendants of} $i$.

\begin{remark}
\rm
Suppose that $A$ has a unique natural basis $B$ and let $E$ be an evolution subalgebra of $A$ having the extension property. Then $E= \spa\left(\{e_i\ \vert \ i \in \Supp (E)\}\right)$.  In particular, the evolution subalgebra having the extension property generated by $e_j \in B$ is equal to $\spa \left(\{e_i \ \vert \ i \in D(j)\cup \{ j\}\rbrace\right)$. Thus, if $e_j$ is algebraically persistent then $e_i$ is algebraically persistent for all $i \in D(j)$. This result can be applied to any of the subalgebras $A_i^0$ appearing in \eqref{descompForm}.

The hypothesis of having a unique natural basis cannot be eliminated. For an example, consider Example \ref{number-transient}. It happens that $1, 2 \in D(3)$ but the subalgebra generated by $e_3$ does not contain neither $e_1$ nor $e_2$.
\end{remark}

Given an evolution algebra with natural basis $B$, we may define the ascendents of any element of $B$, in an analogous way. The transpose $M_B^t$ of   the structure matrix $M_B$ provides direct information about the ascendents of any element of $B$. The mathematical  study of the structure matrix and its transpose may shed new light on the classification of finite dimensional evolution algebras. In particular, it may help if we  compare  different natural bases since properties of the first-generation of descendants  or ascendents may depend on the natural basis (see for instance Example  \ref{number-transient}). 

\begin{definition}
\rm
Let $A$ be an evolution algebra with natural basis $B= \{e_i\}_{i\in \Lambda}$ and let $M_B$ be its structure matrix. The \emph{adjoint algebra of} $A$ relative to a basis $B$, denoted by $A_B^*$,  is the evolution algebra with natural basis $B$ and structure matrix $M_B^t$, the transpose of $M_B$. The definition strongly depends on the basis, as can be seen in the example that follows.
\end{definition}

\begin{example}
\rm
Consider again the evolution algebra in Example \ref{number-transient}. 
Then the structure matrix of $A_B^*$ is
\begin{equation*}
  M_B^t= M_B=\left(
     \begin{array}{ccc}
       1 & 1 & 1 \\
       1 & 1 & 1 \\
       1 & 1 & 0 \\
     \end{array}
   \right),
\end{equation*}  while the structure matrix of $A_{B'}^*$ is \begin{equation*}
  M_{B'}^t=\left(
     \begin{array}{ccc}
       2 & 0 & 2\\
       2& 0 & 2 \\
       1 & 0 & 0 \\
     \end{array}
   \right).
\end{equation*}
Note that the evolutions algebras $A_{B}^*$ and  $A_{B'}^*$ are not isomorphic because the first one has zero annihilator, while the second one has annihilator of dimension 1.  This example also shows that the symmetry of the structure matrix depends on the natural basis. 
%
\end{example}

\begin{lemma}\label{order} Let $A$ be an evolution algebra with natural basis $B$, and let $B'$  be a natural basis obtained by changing the order of elements of $B$. Then $A^*_B$ and $A^*_{B'}$ correspond to the same evolution algebra.
\end{lemma}
\begin{proof}
It follows from \cite[p. 156]{CKS} that there exists $P \in S_n$ such that $$M_{B'}= P^{-1}M_B P^{(2)}=P^{-1}M_B P.$$
Thus $$M_{B'}^t= P^tM_B^t (P^t)^{-1}= P^tM_B^t ((P^t)^{-1})^{(2)}.$$ By the change of basis formula for evolution algebras \cite[p. 30]{T}, we deduce the desired result.
\end{proof}

By a careful observation of  the structure matrix of an evolution algebra and its transpose, we get the following.

\begin{proposition}\label{adjoint}
 Let $A$ be an evolution algebra and let $B$ be a natural basis.  Then:
\begin{enumerate}[\rm (i)]
\item $A$ is irreducible if and only if $A_B^*$ is irreducible.
\item If $B' \subseteq B$ and  $\spa(B')$ is an  evolution subalgebra of $A$, then $ \spa(B\setminus B')$ is an evolution subalgebra of $A_B^*$.
\item $A$ is simple if and only if $A_B^*$ is simple.
\item  $A$ is basic simple if and only if $A_B^*$ is basic simple.
\item $A$ is nilpotent if and only if $A_B^*$ is nilpotent.
\end{enumerate}
\end{proposition}
\begin{proof}
It is straightforward if we apply the characterization of the structure matrix of the algebra in each case. Thus (i), (ii),  and (iv) are immediate. For (iii), use \cite[Corollary 4.6]{CSV1} and observe first that both $A$ and $A_B^*$ have nonsingular structure matrices. By Lemma \ref{order}, we get the same adjoint if we change the order of the elements of $B$. Thus we may reorganize $B$ and write the structure matrix of $A$ in the  following  form 
\begin{equation*}
  \left(
    \begin{array}{cc}
      W & U \\
      0 & Y \\
    \end{array}
  \right).
\end{equation*}
Once again we reorganize the basis and obtain the desired form for $A_B^*$. 
For (v)  we argue analogously and  apply    \cite[Corollary 3.6]{EL2} and Lemma \ref{order}.
\end{proof}

\begin{example}
\rm
Having $A$ a unique natural basis is not equivalent for $A^*_B$ to have a unique natural basis. For an example, just take the following structure matrix and its transpose
$$  \left(
     \begin{array}{ccc}
       1 & 1 & 2\\
       1& 1 & 4 \\
       1 & 1 & 7 \\
     \end{array}
   \right). $$ 
Then $A^*_B$ has a unique natural basis since it satisfies Condition (2LI) while $A$ has not because the first and the second column are linearly dependent, hence $A$ does not satisfies Condition (2LI). 
\end{example}

 A natural question for an evolution algebra which does not have a unique basis is, what is the best  
natural basis that  reveals more the structure of the algebra? The property of the adjoint may help in the choice of the best basis.  It seems that degeneracy of the adjoint is one of the key differences between different natural bases.  

Next we focus on degenerate elements of the adjoint algebra. 

\begin{lemma} Let $A$ be an evolution algebra with natural basis $B$. Then 
   $$\ann (A^*_B)= \spa \left(\{e_i \in B\ \vert \ i \not\in D(j) \mbox { for all j} \}\right).$$
Moreover, $A^2 \cdot \ann(A^*_B)=\lbrace 0\rbrace $, where $\cdot$ denotes the product in the algebra $A$.
\end{lemma}
\begin{proof} 
We have introduced the $\cdot$ in the statement because we may consider products in two different algebras: $A$ and $A^\ast_B$. In what follows we will not use this notation and will multiply always inside $A$.
If $e_i \in \ann (A^*_B)$, then the row $i$ is equal to zero. Thus $i$ cannot be a descendant of any $j$. The converse is immediate.  

Next put $B= \lbrace e_1, \ldots, e_r, e_{r+1}, \ldots, e_n\rbrace$ and suppose that $\ann (A^*_B)= \spa \lbrace e_1, \ldots, e_r\rbrace$. Then $A^2 \subseteq \spa \lbrace e_{r+1}, \ldots, e_n\rbrace$ and therefore $A^2e_j = \lbrace 0\rbrace$ for every $1 \leq j \leq r$. This implies that $A^2 u= \lbrace 0\rbrace$, for every $u \in \ann (A^*_B)$, as desired.
\end{proof}

\begin{remark} 
\rm
Let $A$ be an evolution algebra satisfying $A^3=\lbrace 0\rbrace$. Then it follows from Lemma \ref{nilp-3} that for every natural basis $B$ of $A$, $B= \ann (A) \cup \ann(A^*_B)$. 
\end{remark}

\begin{proposition}\label{AdjTrans}
Let $A$ be an irreducible  evolution algebra. Suppose that there is a natural basis $B$ of $A$ such that $\ann (A^*_B )\neq \lbrace 0\rbrace$. Then the linear space $\ann ( A^*_B)$ is generated by algebraically transient elements of $A$. Moreover, there exists a $0$th decomposition of $A$ such that $\ann (A^*_B )$ is contained in the $0$th-transient space of $A$. 
\end{proposition}
\begin{proof} Since $A$ is irreducible, $\ann(A) \cap \ann (A^*_B)= \lbrace 0\rbrace$. Let $B= \lbrace u_1, \ldots, u_r, u_{r+1}, \ldots, u_n\rbrace$ and suppose that $\ann (A^*_B)= \spa \lbrace u_1, \ldots, u_r\rbrace$. Then the structure matrix of $A^*_B$ has the following block form
\[  \left(
     \begin{array}{ccc}
       0 &  *\\
       0& * \\  
     \end{array}
   \right) \]
   Thus, the structure matrix of $A$ with respect to $B$ has the following block form
   \[  \left(
     \begin{array}{ccc}
       0 &  0\\
       *& G \\  
     \end{array}
   \right) \] In particular, observe that $A$ is not basic simple. Next let $i \in \lbrace1, \ldots, r\rbrace$. Then $u_i\not\in \rm {alg}(u_i^2)$, the subalgebra of $A$ generated by $u_i^2$. But $u_i^2 \neq 0$ and $\rm {alg}(u_i^2) \subseteq \spa \lbrace u_{r+1}, \ldots, u_n\rbrace$, hence  $u_i$ must be algebraically transient. Consider the evolution algebra $A'$ corresponding to the  structure matrix $G$. Then, by considering the $0$th decomposition of all the irreducible subalgebras of $A'$ having the extension property, we may write 
   \[A'= A_1^0 \oplus \ldots \oplus A_s^0  \; \dot{+} \; E'_0, \]
   where $E'_0$ is a linear space spanned by transient elements of $A'$. Then  The $0$th transient space of $A$ is $E_0= \ann (A^*_B)+E'_0$,  and
   \[A= A_1^0 \oplus \ldots \oplus A_s^0  \; \dot{+} \; E'_0 \; \dot{+}\ \ann (A^*_B)\].
\end{proof}

\par


\noindent

\begin{thebibliography}{99}

\bibitem{CCY} Yolanda Cabrera Casado, \textit{Evolution algebras. Doctoral dissertation. Universidad de M\'alaga} (2016).  \url{http://hdl.handle.net/10630/14175}
%
\bibitem{CKS} Yolanda Cabrera Casado, Muge Kanuni, and Mercedes Siles Molina, {Basic ideals in evolution algebras}, \textit{Linear Alg.  and Appl.} \textbf{570} (2019), 148--180.

\bibitem{CGOT} Luisa M. Camacho, Jos\'e  Ram\'on G\'{o}mez, Bakhrom A. Omirov and Rustam M. Turdibaev, {Some properties of evolution algebras}, \textit{Bull. Korean Math. Soc.} \textbf{50} (5), (2013), 1481--1494.
%
\bibitem{CSV1} Yolanda Cabrera Casado, Mercedes Siles Molina, and M. Victoria Velasco, {Evolution algebras of arbitrary dimension and their decompositions}, \textit{Linear Alg. and Appl.} \textbf{495}(2016), 122--162.
%
%
\bibitem{CSV2} Yolanda Cabrera Casado, Mercedes Siles Molina, and M. Victoria Velasco, {Classification of three dimensional evolution algebras}, \textit{Linear Alg. and Appl.} \textbf{524}
(2017), 68--108.
%
\bibitem{CGMMS1} Maria Inez Cardoso Gon{\c c}alvez, Daniel Gon{\c c}alvez, Dolores Mart\'in Barquero, C\'andido Mart\'in Gonz\'alez and Mercedes Siles Molina. {Squares and associative representations of two dimensional evolution algebras.} ArXiv:1807.02362 (2018).
%
%
\bibitem{EL1} Alberto Elduque, and Alicia Labra, {Evolution algebras and graphs}, \textit{J. Algebra Appl.}  \textbf{14} (7) (2015), 1550103, 10 pp.
%
\bibitem{EL2} Alberto Elduque, and Alicia Labra, {On nilpotent evolution algebras}, \textit{Linear Algebra Appl.}  \textbf{505} (2016), 11--31.
%


%
\bibitem{T} Jianjun P. Tian, \textit{Evolution algebras and their applications}. Lecture Notes in Mathematics \textbf{1921}, Springer-Verlag, Berlin, 2008.
%
\bibitem{T2} Jianjun P. Tian, { Invitation to research of new mathematics from biology: Evolution algebras}, Topics in Functional Analysis and Algebra. Contemp. Math., 672, 257-272, Amer. Math. Soc., Providence, RI, 2016. 
%
\bibitem{VT} Jianjun P. Tian and Petr Vojtechovsky, Mathematical concept of evolution algebras in non-Mendelian
genetics, \textit{Quasigroup Relat. Syst.} \textbf{14} (2006),  111-122.

\bibitem{ZSSS}  K. A. Zhevlakov, Arkadii M. Slin'ko, Ivan P. Shestakov $\and$ A.I. Shirshov
 \textit{Rings that are nearly associative}.   Academic Press, New York, 1982.
\end{thebibliography}
\end{document}